\pdfoutput=1
\documentclass[11pt, reqno]{amsart}
\usepackage[dvips]{graphicx}
\usepackage{bmpsize} 
\usepackage[margin=1.12in]{geometry}

\usepackage{mathrsfs}
\usepackage{amsmath,amscd}
\usepackage{amssymb}
\usepackage{amscd}
\graphicspath{/home/jhois/Documents/Images/}  

\newtheorem{theorem}{Theorem}[section]
\newtheorem{proposition}[theorem]{Proposition}
\newtheorem{lemma}[theorem]{Lemma}

\theoremstyle{definition} 

\newcommand{\R}{\mathbb R}

\numberwithin{equation}{section} 
\numberwithin{theorem}{section}
\numberwithin{figure}{section}

\begin{document}

\author[J.A.~Hoisington]{Joseph Ansel Hoisington} \address{Max Planck Institute for Mathematics}\email{hoisington@mpim-bonn.mpg.de}  

\title[Energy-minimizing Maps of $\R P^{n}$]{Energy-minimizing Mappings of Real Projective Spaces}

\keywords{Energy-minimizing maps, infima of energy functionals in homotopy classes}
\subjclass[2020]{Primary 53C43 Secondary 53C35}

\begin{abstract}
We give a sharp lower bound for the energy in homotopy classes of mappings from real projective space to Riemannian manifolds, together with an upper bound for its infimum.  We characterize the maps which attain this lower bound for energy, and we explain how the infimum of the energy in a homotopy class of mappings of real projective $n$-space is determined by an associated class of mappings of the real projective plane.  
\end{abstract}
\maketitle


\section{Introduction} 
\label{introduction} 

A basic observation in the study of harmonic maps is that the identity map of the sphere $(S^{n},g_{0})$ of dimension $n \geq 3$ does not minimize energy in its homotopy class.  Eells and Sampson illustrate this, with an example which they attribute to Morrey, in their foundational work on the subject in \cite{ES1}.  It is now known to be a special case of several more general results.  One of these occurs in the work of White \cite{Wh1}, which implies that in any homotopy class of mappings from quaternionic projective space, the Cayley projective plane, or the sphere of dimension $3$ or greater to a Riemannian manifold, the infimum of the energy is $0$. \\ 

The results in \cite{H1} show that in all homotopy classes of mappings from complex projective space to a Riemannian manifold, the infimum of the energy is proportional to the infimal area in the homotopy class of mappings of $S^{2}$ which represents the induced homomorphism on the second homotopy group.  Together, these results and the results in \cite{Wh1} above determine the infimum of the energy in all homotopy classes of mappings of compact, rank-$1$ Riemannian symmetric spaces other than real projective space (in particular, for all simply-connected spaces in this family).  The purpose of this note is to address this problem for real projective space.  Our first result is a two-sided estimate for the infimum of the energy in a homotopy class: 

\begin{theorem}
\label{inf_estimate}
Let $(\R P^{n},g_{0})$, $n \geq 2$, be real projective space with its Riemannian metric of constant curvature $1$.  Let $\Psi$ be a homotopy class of mappings from $(\R P^{n},g_{0})$ to a Riemannian manifold $(M,g)$ and $\psi$ the homotopy class of mappings of $(\R P^{2},g_{0})$ represented by composing the inclusion $\R P^{2} \subseteq \R P^{n}$ with $F \in \Psi$.  Let $A^{\star}$ be the infimal area of mappings $f \in \psi$.  Then, letting $E_{2}(F)$ be the energy of a map $F$, $\sigma(n)$ the volume of the unit $n$-sphere, and $C_{n} = \frac{n \sigma(n)}{8 \pi}$, 

\begin{equation}
\label{inf_est_eqn_2}
\displaystyle C_{n} A^{\star} \leq \inf\limits_{F \in \Psi} E_{2}(F) \leq 2 \left(\frac{n-1}{n}\right) C_{n} A^{\star}. \smallskip    
\end{equation}

If $n \geq 3$ and $F \in \Psi$ has $E_{2}(F) = C_{n}A^{\star}$, then $F$ is a homothety onto a totally geodesic submanifold. \end{theorem}

The case $n=2$ of Theorem \ref{inf_estimate} states that the infimum of the energy in any homotopy class of mappings of $(\R P^{2},g_{0})$ is equal to the infimal area.  In Theorem \ref{lemaire_theorem} we quote a result of Lemaire \cite{Lem1} which implies that maps of $(\R P^{2},g_{0})$ that attain this infimal value are conformal branched immersions. \\ 

The lower bound $C_{n}A^{\star}$ in (\ref{inf_est_eqn_2}) is equal to the energy of the identity map of $(\R P^{n},g_{0})$ in its homotopy class, so this part of Theorem \ref{inf_estimate} is optimal.  Croke first established that the identity map of $(\R P^{n},g_{0})$ minimizes energy in its homotopy class in \cite{Cr1}, as a corollary of:  

\begin{theorem}{\em (Croke \cite[Theorem 1]{Cr1})} \label{chris_theorem} Let $F:(\R P^{n},g_{0}) \rightarrow (M,g)$ be a map to a Riemannian manifold, $\gamma$ a generator of $\pi_{1}(\R P^{n})$, and $L^{\star}$ the infimal length of curves freely homotopic to $F_{*}\gamma$ in $M$.  Then:  

\begin{equation}
\label{chris_equation}
\displaystyle E_{2}(F) \geq \frac{n\sigma(n)}{4\pi^{2}} L^{\star^{2}}. \smallskip
\end{equation}

Equality implies that $F$ is a homothety onto a totally geodesic submanifold. 
\end{theorem}

Theorem \ref{chris_theorem} follows from Theorem \ref{inf_estimate} and Pu's systolic inequality \cite{Pu,CK1}, which implies that the lower bound in Theorem \ref{inf_estimate} is bounded below by the lower bound in Theorem \ref{chris_theorem}.  The strengthened versions of Pu's inequality given by Katz-Nowik and Katz-Sabourau in \cite{KN1,KS1} imply that in most homotopy classes, i.e. other than in those for which the induced class of mappings of $(\R P^{2},g_{0})$ admits a sequence of maps with a fairly rigid characterization, the lower bound in Theorem \ref{inf_estimate}, and thus the infimum of the energy, is strictly greater than the lower bound in Theorem \ref{chris_theorem}.  The result of Theorem \ref{inf_estimate} also shows that the infimum of the energy in a homotopy class cannot be bounded above in terms of the infimal length $L^{\star}$ in Theorem \ref{chris_theorem}. \\ 

Although Theorem \ref{inf_estimate} generally does not give the exact value of the infimal energy in a homotopy class, it complements the results of White in \cite{Wh1}, which imply that the infimum of the energy in a homotopy class of mappings of real projective $n$-space depends only on the induced class of mappings of the real projective plane: 

\begin{theorem}{\em (See \cite[Theorem 1]{Wh1})}
\label{white_result}
Let $\Psi$ and $\psi$ be as in Theorem \ref{inf_estimate}, and let $\widehat{\Psi}$ be the union of all homotopy classes of mappings from $(\R P^{n},g_{0})$ to $(M,g)$ that induce the class $\psi$ of mappings of $(\R P^{2},g_{0})$.  Then $\inf\limits_{F \in \Psi} E_{2}(F) = \inf\limits_{F \in \widehat{\Psi}} E_{2}(F)$. 
\end{theorem} 

Theorem \ref{white_result} implies that the infimum of the energy is zero in any homotopy class of mappings of $(\R P^{n},g_{0})$ that induces the trivial class of mappings of $(\R P^{2},g_{0})$.  Theorem \ref{inf_estimate} sharpens this statement by showing that the infimum of the energy in a homotopy class of mappings of $(\R P^{n},g_{0})$ is zero if and only if it is zero in the induced class of mappings of $(\R P^{2},g_{0})$. \\ 

After proving Theorem \ref{inf_estimate} we give a corollary of the proof, in Proposition \ref{hyperplane_estimate}, which further clarifies the geometric data that determine the infimum of the energy in a homotopy class of mappings of real projective space.  The result implies that the infimum of the energy in a homotopy class of mappings of $(\R P^{n},g_{0})$ for $n \gg 2$ is approximately proportional to the infimum in the induced class of mappings of $(\R P^{n-1},g_{0})$. 

\subsection*{Outline and Notation:} In Section \ref{prelims}, we will review some properties of the energy functional and harmonic maps and establish some results needed in the proof of Theorem \ref{inf_estimate}.  We will also explain how Theorem \ref{white_result} follows from the results in \cite{Wh1}.  In Section \ref{main_proofs}, we will prove Theorem \ref{inf_estimate} and discuss some corollaries of the proof.  We will write $\sigma(n)$ for the volume of the unit $n$-sphere, $g_{0}$ for the canonical Riemannian metric on $\R P^{n}$ and its covering $S^{n}$, and $E_{2}(F)$ for the energy of any map $F$ between Riemannian manifolds. 

\subsection*{Acknolwedgements:} I am happy to thank Werner Ballmann, Christopher Croke, and Sergio Zamora for their input and feedback about this work, and the Max Planck Institute for Mathematics for support and hospitality. 


\section{Preliminary Results}  
\label{prelims} 


The energy of a Lipschitz map $F:(N^{n},h) \rightarrow (M^{m},g)$ of Riemannian manifolds is: 

\begin{equation}
\label{energy_eqn}
\displaystyle E_{2}(F) = \frac{1}{2} \int\limits_{N} |dF_{x}|^{2} dVol_{h}, \smallskip 
\end{equation}
where $|dF_{x}|$ is the Euclidean norm of $dF:T_{x}N \rightarrow T_{F(x)}M$ at a point $x \in N$ at which $F$ is differentiable (note that the definition of energy in some papers, including \cite{Wh1} and \cite{Cr1}, differs from (\ref{energy_eqn}) by a constant multiple).  We will discuss some background related to maps which are critical for energy, called harmonic maps, below.  The starting point for our proof of Theorem \ref{inf_estimate} is a formula for the energy of a map due to Croke: 

\begin{proposition}{\em (Croke \cite[Proposition 1]{Cr1})} \label{chris_lemma} Let $F:(N^{n},h) \rightarrow (M^{m},g)$ be a Lipschitz map of Riemannian manifolds, $U(N,h)$ the unit tangent bundle of $(N,h)$, and $d\vec{u}$ the canonical measure on $U(N,h)$.  Then: 
	
\begin{equation}
\displaystyle E_{2}(F) = \frac{n}{2 \sigma(n-1)} \int\limits_{U(N,h)} |dF(\vec{u})|^{2} d\vec{u}. 
\end{equation} 
\end{proposition}

For mappings of $(\R P^{n},g_{0})$, Proposition \ref{chris_lemma} implies: 

\begin{lemma}
\label{double_fibration_prop}
Let $F:(\R P^{n},g_{0}) \rightarrow (M,g)$ be a Lipschitz map to a Riemannian manifold, $Gr(k)$ the set of totally geodesic $\R P^{k}$ in $\R P^{n}$, and for $\mathcal{Q} \in Gr(k)$, let $F|_{\mathcal{Q}}$ be the composition of the inclusion $\mathcal{Q} \subseteq \R P^{n}$ with $F$.  Let $d\mathcal{Q}$ be the measure on $Gr(k)$ which is invariant under the action of the isometry group $Isom(\R P^{n},g_{0})$ of $(\R P^{n},g_{0})$, normalized to have total volume $\frac{n\sigma(n)}{k\sigma(k)}$.  Then:   

\begin{equation}
\displaystyle E_{2}(F) = \int\limits_{Gr(k)} E_{2}(F|_{\mathcal{Q}}) d\mathcal{Q}. 
\end{equation}
\end{lemma}

\begin{proof} Let $\mathcal{V}_{k}$ denote the set of pairs $(\vec{u},\mathcal{Q})$, where $\vec{u}$ belongs to the unit tangent bundle $U(\R P^{n},g_{0})$, $\mathcal{Q} \in Gr(k)$, and $\vec{u}$ is tangent to $\mathcal{Q}$, and let $d\mu_{k}$ be the $Isom(\R P^{n},g_{0})$-invariant measure on $\mathcal{V}_{k}$, normalized to have total volume $\frac{n\sigma(n)}{4}$.  Calculating via the $Isom(\R P^{n},g_{0})$-equivariant fibrations $(\vec{u},\mathcal{Q}) \mapsto \vec{u}$ of $\mathcal{V}_{k}$ over $U(\R P^{n},g_{0})$ and $(\vec{u},\mathcal{Q}) \mapsto \mathcal{Q}$ of $\mathcal{V}_{k}$ over $Gr(k)$, Proposition \ref{chris_lemma} implies: 

\begin{equation}
\label{double_fibration_pf_eqn_1}
\displaystyle E_{2}(F) = \int\limits_{\mathcal{V}_{k}} |dF(\vec{u})|^{2} d\mu_{k} = \int\limits_{Gr(k)} E_{2}(F|_{\mathcal{Q}}) d\mathcal{Q}. 
\end{equation}
\end{proof} 

The case $k=1$ of Lemma \ref{double_fibration_prop}, in which $Gr(1)$ is the space of oriented geodesics in $(\R P^{n},g_{0})$, is simpler than the general result: each $\vec{u} \in U(\R P^{n},g_{0})$ is tangent to a unique geodesic, so $U(\R P^{n},g_{0})$ admits an $Isom(\R P^{n},g_{0})$-equivariant fibration over $Gr(1)$ and the calculation via the ``incidence variety" $\mathcal{V}_{k}$ is unnecessary.  This observation is the starting point for the proof of Theorem \ref{chris_theorem} in \cite{Cr1}. 

\begin{lemma}
\label{inf_upper_bound_prop}
Let $\Psi$ be as in Theorem \ref{inf_estimate}.  Let $\Psi'$ be the homotopy class of mappings of $(\R P^{n-1},g_{0})$ represented by composing the inclusion $\R P^{n-1} \subseteq \R P^{n}$ with $F \in \Psi$ and $\beta$ the infimum of the energy in $\Psi'$.  Then for $n \geq 3$, $\inf\limits_{F \in \Psi} E_{2}(F) \leq \frac{\sigma(n-2)}{\sigma(n-3)} \beta$.
\end{lemma}

\begin{proof} Let $\mathcal{Q}_{0}$ be a fixed totally geodesic $\R P^{n-1} \subseteq \R P^{n}$ and $x_{0} \in \R P^{n}$ the unique point at distance $\frac{\pi}{2}$ from $\mathcal{Q}_{0}$.  Let $\tau:S^{n} \rightarrow \R P^{n}$ be the covering map, $\widetilde{x}_{0}$ a point in $\tau^{-1}(x_{0})$, and $\xi:S^{n} \setminus \lbrace -\widetilde{x}_{0} \rbrace \rightarrow \R^{n}$ the stereographic projection which takes $\widetilde{x}_{0}$ to the origin.  Let $\theta_{t}$ be the conformal diffeomorphism of $S^{n}$ given by multiplication by $t$ in the coordinates defined by $\xi$.  An elementary calculation shows that for $n \geq 3$, $\lim\limits_{t \rightarrow \infty} E_{2}(\theta_{t}) = 0$.  (This is the example given by Eells and Sampson in \cite{ES1} to show that for $n \geq 3$, the identity map of $(S^{n},g_{0})$ is homotopic to maps with arbitrarily small energy.) \\ 
	
Let $\Omega \subseteq S^{n}$ be the open hemisphere centered at $\widetilde{x}_{0}$, regarded as a fundamental domain for the covering $\tau$.  Define $\Theta_{t}:\R P^{n} \rightarrow \R P^{n}$ as follows:  if $x \in \R P^{n}$ has a preimage $\widetilde{x}$ via $\tau$ in $\theta_{t}^{-1}(\Omega)$, then $\Theta_{t}(x) = \tau \circ \theta_{t}(\widetilde{x})$.  Otherwise, $\Theta_{t}(x)$ is the image of $x$ under the nearest-point retraction of $\R P^{n} \setminus \lbrace x_{0} \rbrace$ onto $\mathcal{Q}_{0}$.  Because $\theta_{t}$ is conformal and $\lim\limits_{t \rightarrow \infty} E_{2}(\theta_{t}) = 0$, for any Lipschitz map $F$, the energy of $F \circ \Theta_{t}$ on the domain $\tau(\theta_{t}^{-1}(\Omega))$ goes to $0$ as $t \rightarrow \infty$.  Letting $F'$ be the map of $\mathcal{Q}_{0}$ given by composing the inclusion $\mathcal{Q}_{0} \subseteq \R P^{n}$ with $F$, a calculation in Fermi coordinates about $\mathcal{Q}_{0}$ shows: 

\begin{equation}
\label{inf_upper_bound_prop_pf_eq_1}
\displaystyle \lim\limits_{t \rightarrow \infty} E_{2}(F \circ \Theta_{t}) = \int\limits_{\mathcal{Q}_{0}} \int\limits_{0}^{\frac{\pi}{2}} \cos^{n-3}(r) |dF'|^{2} dr dx = \frac{\sigma(n-2)}{\sigma(n-3)} E_{2}(F'). \smallskip 
\end{equation}

For any $F \in \Psi$ and $\varepsilon > 0$, the induced map $F'$ of $\mathcal{Q}_{0}$ is homotopic to a map $\underline{F'}$ with $E_{2}(\underline{F'}) < \beta + \varepsilon$.  By the homotopy extension property for $\R P^{n-1} \subseteq \R P^{n}$, $F$ is homotopic to a map $\underline{F}$ of $\R P^{n}$ which coincides with $\underline{F'}$ along $\mathcal{Q}_{0}$.  The result then follows from (\ref{inf_upper_bound_prop_pf_eq_1}) by composing $\underline{F}$ with the deformation $\Theta_{t}$. \end{proof}

The proof of Theorem \ref{inf_estimate} depends on a basic fact about mappings of surfaces, which we record here for later reference: 

\begin{lemma}
\label{surface_lemma}

Let $f$ be a map from a Riemannian surface $(\Sigma,h)$ to a manifold $(M,g)$.  Then the integrand $\frac{1}{2}|df|^{2}$ for the energy of $f$ is bounded below pointwise by the integrand $\sqrt{\det(df^{T} \circ df)}$ for the area of the image of $f$, with equality precisely where $f$ is semiconformal, i.e. where $f^{*}g$ is a scalar multiple of $h$. 
\end{lemma}

Lemma \ref{surface_lemma} implies that, in the notation of Theorem \ref{inf_estimate}, the infimal area $A^{\star}$ is a lower bound for the energy of maps $f \in \psi$.  Our next lemma shows it is the infimum: 

\begin{lemma}
\label{2-sphere_lemma}

Let $f:(S^{2},g_{0}) \rightarrow (M,g)$ be a smooth map to a Riemannian manifold and $A(f)$ the area of its image.  Then for any $\delta > 0$ there is a diffeomorphism $\phi: S^{2} \rightarrow S^{2}$, homotopic to the identity, such that $E_{2}(f \circ \phi) < A(f) + \delta$.  In particular, $f$ is homotopic to a map with energy less than $A(f) + \delta$.  If $f$ is antipodally invariant, we can choose $\phi$ to be antipodally invariant, so that the same conclusions hold for maps of $(\R P^{2},g_{0})$. 
\end{lemma}

\begin{proof} If $f$ is an immersion, the result follows immediately from the uniformization theorem.  If not, let $f_{r}:(S^{2},g_{0}) \rightarrow (M,g) \times (S^{2}, rg_{0})$ be the product of $f$ with a homothety from $(S^{2},g_{0})$ to a $2-$sphere $(S^{2},rg_{0})$ of constant curvature $\frac{1}{r}$.  The uniformization theorem implies there is a diffeomorphism $\phi_{r}:S^{2} \rightarrow S^{2}$ homotopic to the identity such that $f_{r} \circ \phi_{r}$ is conformal.  Letting $A(f_{r})$ be the area of the image of $f_{r}$, Lemma \ref{surface_lemma} then implies $E_{2}(f_{r} \circ \phi_{r}) = A(f_{r})$.  An elementary calculation shows that $E_{2}(f \circ \phi_{r}) < E_{2}(f_{r} \circ \phi_{r})$, and that by choosing $r$ small enough we can ensure $A(f_{r}) < A(f) + \delta$ for any $\delta > 0$.  If $f$ is antipodally invariant, the uniformization theorem implies we can choose an antipodally invariant $\phi_{r}$ as well. \end{proof} 

A homotopy class may not contain an energy-minimizing map, as discussed by Sacks and Uhlenbeck in \cite{SU}, but if a map of $(\R P^{2},g_{0})$ minimizes energy in its homotopy class, Lemmas \ref{surface_lemma} and \ref{2-sphere_lemma} imply it is semiconformal.  Lemaire showed that all harmonic maps of $(\R P^{2},g_{0})$ have this property: 

\begin{theorem}{\em (Lemaire \cite[Theorem 2.8]{Lem1})} \label{lemaire_theorem} A harmonic map from $(S^{2},g_{0})$ or $(\R P^{2},g_{0})$ to a Riemannian manifold is a conformal branched immersion. \end{theorem}

Although Lemma \ref{surface_lemma} implies that conformal maps between surfaces are harmonic, maps of manifolds of dimension $n \geq 3$ which are both harmonic and conformal must be homotheties.  For completeness, we prove this here: 

\begin{lemma}
\label{harmonic_conformal_lemma}
Let $F:(N^{n},h) \rightarrow (M^{m},g)$ be a smooth immersion of Riemannian manifolds, and suppose $F^{*}g = e^{2\eta}h$, where $\eta$ is a smooth function on $N$.  If $n = 2$, $F$ is harmonic if and only if the image of $F$ is a minimal submanifold of $(M,g)$.  If $n \geq 3$, $F$ is harmonic if and only if $F$ is a homothety and the image of $F$ is a minimal submanifold of $(M,g)$.  
\end{lemma} 

\begin{proof} Let $\nabla^{h}$ be the connection of the metric $h$ and $\nabla^{h}\eta$ the gradient of $\eta$ relative to $h$, and let $\nabla^{*}$ be the connection of the metric $F^{*}g$ on $N$.  An elementary calculation shows that for vector fields $V,W$ on $N$, 
	
\begin{equation}
\label{harm_conf_pf_eqn_1}
\displaystyle \nabla^{*}_{V}W = \nabla^{h}_{V}W + V(\eta)W + W(\eta)V - h(V,W)\nabla^{h}\eta. \smallskip 
\end{equation}

Let $\nabla^{g}$ be the connection of the metric $g$ on $M$, and let $\alpha_{F}$ be second fundamental form of $F$; that is, the $F^{*}TM$-valued, symmetric bilinear form which, for vector fields $V,W$ on $N$, satisfies $\alpha_{F}(V,W) = F^{*}\nabla^{g}_{V}F_{*}W - F_{*}\nabla^{h}_{V}W$.  We can also write $\alpha_{F}$ as:   

\begin{equation}
\label{harm_conf_pf_eqn_2}
\displaystyle \alpha_{F}(V,W) = (\nabla^{g}_{F_{*}V}F_{*}W)^{\perp} + F_{*}\left(\nabla^{*}_{V}W - \nabla^{h}_{V}W\right), \smallskip 
\end{equation}
where $(\nabla^{g}_{F_{*}V}F_{*}W)^{\perp}$ is the component of $\nabla^{g}_{F_{*}V}F_{*}W$ normal to $F(N)$.  Taking the trace of (\ref{harm_conf_pf_eqn_2}) using (\ref{harm_conf_pf_eqn_1}), and writing $\vec{H}$ for the mean curvature vector of $Im(F) \subseteq (M,g)$, gives: 

\begin{equation}
\label{harm_conf_eqn}
\displaystyle Tr(\alpha_{F}) = (2-n)\nabla^{h}\eta + e^{2\eta}\vec{H}. \smallskip 
\end{equation}

By Eells and Sampson \cite{ES1}, a map $F$ is harmonic if and only if $Tr(\alpha_{F}) = 0$.  \end{proof}

We finish this section by briefly explaining how Theorem \ref{white_result} follows from White's results in \cite{Wh1}:  In \cite{Wh1}, maps of a closed Riemannian manifold $(N,h)$ are defined to be $2$-homotopic if their restrictions to the $2$-skeleton of a triangulation of $N$ are homotopic.  The result of \cite[Theorem 1]{Wh1} implies that the infimum of the energy in the homotopy class of a map $F$ is equal to the infimum in its $2$-homotopy class (this is explicitly stated for maps between closed Riemannian manifolds in \cite[Corollary 1]{Wh1} but follows from \cite[Theorem 1]{Wh1} for maps from a closed Riemannian manifold to any manifold).  Although $\R P^{2}$ is not the $2$-skeleton of a triangulation of $\R P^{n}$ (unless $n=2$), it is straightforward to check that $\widehat{\Psi}$ in Theorem \ref{white_result} coincides with the $2$-homotopy class of maps $F \in \Psi$. 


\section{Upper and Lower Bounds for the Infimum of the Energy}  
\label{main_proofs}   


Below, we will prove Theorem \ref{inf_estimate} and then discuss some corollaries of the proof and alternate approaches to this result. 

\begin{proof}[Proof of Theorem \ref{inf_estimate}] For $n=2$, Theorem \ref{inf_estimate} follows from Lemmas \ref{surface_lemma} and \ref{2-sphere_lemma}.  For $n \geq 3$, let $\beta$ be as in Lemma \ref{inf_upper_bound_prop}.  The case $k = n-1$ of Lemma \ref{double_fibration_prop} implies: 
	
\begin{equation}
\label{inf_est_pf_eqn_1}
\displaystyle \frac{n \sigma(n)}{(n-1)\sigma(n-1)} \beta \leq \inf\limits_{F \in \Psi}E_{2}(F). \smallskip 
\end{equation}

Successive application of this estimate, together with the result in the case $n=2$, gives the lower bound $C_{n} A^{\star}$ for the infimal energy in Theorem \ref{inf_estimate}.  Likewise, successive application of Lemma \ref{inf_upper_bound_prop} gives the upper bound $\frac{\sigma(n-2)}{2}A^{\star} = 2(\frac{n-1}{n})C_{n}A^{\star}$ for the infimum of the energy in $\Psi$. \\ 

If a map $F \in \Psi$ has $E_{2}(F) = C_{n}A^{\star}$, then $F$ is harmonic, and therefore smooth (cf. \cite[Ch.10]{Aub1}).  For $n \geq 3$, the case $k=2$ of Lemma \ref{double_fibration_prop} also implies that for all $\mathcal{Q} \in Gr(2)$, $F|_{\mathcal{Q}}$ minimizes energy in the homotopy class $\psi$ and is therefore a conformal branched immersion by Theorem \ref{lemaire_theorem}.  Any orthonormal pair of vectors $\vec{e}_{1}, \vec{e}_{2} \in T_{x}\R P^{n}$ is tangent to a totally geodesic $\mathcal{Q}(\vec{e}_{1},\vec{e}_{2}) \cong \R P^{2}$.  Because $F|_{\mathcal{Q}(\vec{e}_{1},\vec{e}_{2})}$ is a conformal branched immersion, $|dF(\vec{e}_{1})| = |dF(\vec{e}_{2})|$, and because this holds for all orthonormal pairs $\vec{e}_{1},\vec{e}_{2}$, the norm $|dF(\vec{e})|$ is constant on unit tangent vectors $\vec{e}$ to $\R P^{n}$ at $x$.  This implies $F$ is semiconformal, and Lemma \ref{harmonic_conformal_lemma} then implies $F$ is a homothety.  The harmonicity of $F|_{\mathcal{Q}}$ for $\mathcal{Q} \in Gr(2)$ also implies that its image $F(\mathcal{Q})$ is a minimal submanifold of $(M,g)$.  For any orthonormal triple of vectors $\vec{e}_{1}, \vec{e}_{2}, \vec{e}_{3}\in T_{x}\R P^{n}$ and $i,j \in \lbrace 1,2,3 \rbrace$, because $\mathcal{Q}(\vec{e}_{i},\vec{e}_{j})$ is totally geodesic in $(\R P^{n},g_{0})$ and its image via $F$ is minimal in $(M,g)$, the second fundamental form $B$ of $Im(F) \subseteq (M,g)$ satisfies $B(\vec{e}_{i},\vec{e}_{i}) +  B(\vec{e}_{j},\vec{e}_{j}) = 0$.  Because this holds for all $i,j \in \lbrace 1,2,3 \rbrace$ and triples $\vec{e}_{1}, \vec{e}_{2}, \vec{e}_{3} \in T_{x}\R P^{n}$, we have $B \equiv 0$, and the image of $F$ is totally geodesic in $(M,g)$. \end{proof}

As an alternative to the iterated use of Lemma \ref{inf_upper_bound_prop} in the proof of Theorem \ref{inf_estimate}, one could choose a totally geodesic $\mathcal{Q} \cong \R P^{2}$ in $\R P^{n}$ and a map $F \in \Psi$ with $E_{2}(F|_{\mathcal{Q}}) < A^{\star} + \varepsilon$, let $\mathcal{R}$ be the unique totally geodesic $\R P^{n-3}$ at maximal distance from $\mathcal{Q}$, and deform the identity map of $\R P^{n}$ along geodesic segments from $\mathcal{R}$ to $\mathcal{Q}$, expanding a tubular neighborhood $\mathcal{T}$ of $\mathcal{R}$ onto $\R P^{n} \setminus \mathcal{Q}$ and retracting $\R P^{n} \setminus \mathcal{T}$ onto $\mathcal{Q}$.  One can give a proof along these lines, however the result is the same as from the argument above.  Likewise, the lower bound $C_{n}A^{\star}$ for the energy of maps $F \in \Psi$ also follows directly from the case $k = 2$ of Lemma \ref{double_fibration_prop}. \\ 

Theorems \ref{inf_estimate} and \ref{chris_theorem} are the only quantitative lower bounds for the energy of a map of real projective space known to the author, however letting $\Psi^{(m)}$ be the homotopy class of mappings of $\R P^{n-m}$ induced by a homotopy class of mappings $\Psi$ of $\R P^{n}$, a stronger lower bound for the energy of maps in $\Psi^{(m)}$, together with the case $k = n-m$ of Lemma \ref{double_fibration_prop}, would give a stronger lower bound for the energy of maps in $\Psi$.  Likewise, for $n-m \geq 3$, if the infimum of the energy in $\Psi^{(m)}$ were known to be less than the upper bound $2(\frac{n-m-1}{n-m})C_{n-m}A^{\star}$ in Theorem \ref{inf_estimate}, an $m$-fold application of Lemma \ref{inf_upper_bound_prop} would give a stronger upper bound for the infimal energy of maps in $\Psi$.  In general, Lemma \ref{inf_upper_bound_prop} and (\ref{inf_est_pf_eqn_1}) in the proof of Theorem \ref{inf_estimate} imply: 

\begin{proposition}
\label{hyperplane_estimate}

Let $\Psi$ be a homotopy class of mappings of $(\R P^{n},g_{0})$, $n \geq 3$, and let $\beta$ be as in Lemma \ref{inf_upper_bound_prop}.  Then, letting $D_{n} = \frac{n\sigma(n)}{(n-1)\sigma(n-1)}$,  

\begin{equation} 
\label{hypersurface_est_eqn}
\displaystyle D_{n} \beta \leq \inf\limits_{F \in \Psi} E_{2}(F) \leq \frac{(n-1)^{2}}{n(n-2)} D_{n} \beta. 
\end{equation}
\end{proposition}

As $n \rightarrow \infty$, the ratio $\frac{(n-1)^{2}}{n(n-2)}$ between the upper and lower bounds in (\ref{hypersurface_est_eqn}) approaches $1$, and the infimum of the energy in a homotopy class of mappings of real projective $n$-space is approximately proportional to the infimum in the induced class of mappings of a hyperplane. 



\begin{thebibliography}{ABCDE}

\bibitem[Aub13]{Aub1} Thierry Aubin: {\em Some Nonlinear Problems in Riemannian Geometry}, Springer Science and Business Media, 2013.

\bibitem[Cr87]{Cr1} Christopher B. Croke: {\em Lower Bounds on the Energy of Maps}, Duke Mathematical Journal 55.4 (1987), 901-908.

\bibitem[CK03]{CK1} Christopher B. Croke and Mikhail Katz: {\em Universal Volume Bounds in Riemannian Manifolds}, Surveys in Differential Geometry 8.1 (2003), 109-137.

\bibitem[ES64]{ES1} James Eells and Joseph H. Sampson: {\em Harmonic Mappings of Riemannian Manifolds}, American Journal of Mathematics 86.1 (1964), 109-160.

\bibitem[H23]{H1} Joseph Hoisington: {\em Energy-minimizing Mappings of Complex Projective Spaces}, ArXiv preprint, arxiv:2311.08285 (2023).

\bibitem[KN20]{KN1} Mikhail Katz and Tahl Nowik: {\em A Systolic Inequality With Remainder in the Real Projective Plane}, Open Mathematics 18.1 (2020), 902-906.

\bibitem[KS21]{KS1} Mikhail Katz and St\'ephane Sabourau {\em A Pu-Bonnesen Inequality}, Journal of Geometry 112.2 (2021), 18.

\bibitem[Lem78]{Lem1} Luc Lemaire: {\em Applications Harmoniques de Surfaces Riemanniennes}, Journal of Differential Geometry 13.1 (1978), 51-78.

\bibitem[Pu52]{Pu} Pao Ming Pu: {\em Some Inequalities in Certain Nonorientable Riemannian Manifolds}, Pacific J. Math 2.1 (1952), 55-71.

\bibitem[SU81]{SU} Jonathan Sacks and Karen Uhlenbeck: {\em The Existence of Minimal Immersions of 2-spheres}, Annals of Mathematics 113.1 (1981), 1-24.

\bibitem[Wh86]{Wh1} Brian White: {\em Infima of Energy Functionals in Homotopy Classes of Mappings}, Journal of Differential Geometry 23.2 (1986), 127-142.

\end{thebibliography}
\end{document}